\documentclass[12pt]{amsart}

\usepackage{amssymb,amsmath}
\usepackage{amssymb,latexsym}
\usepackage{amsfonts}
\usepackage{amssymb}
\usepackage{stmaryrd}
\usepackage{longtable}
\usepackage{amsmath, geometry, amssymb} 
\usepackage{chngcntr}
\counterwithin{table}{section}
\numberwithin{equation}{section}
\newtheorem{theorem}{Theorem}[section]

\newtheorem{corollary}{Corollary}[section]

\newcommand{\sqr}[2]{{\vcenter{\vbox{\hrule height#2pt
                \hbox{\vrule width#2pt height#1pt \kern#1pt
                \vrule width#2pt}\hrule height#2pt}}}}

\newcommand{\beq}{\begin{equation}}
\newcommand{\eeq}{\end{equation}}
\newcommand{\beqar}{\begin{eqnarray}}
\newcommand{\eeqar}{\end{eqnarray}}
\def\beqars{\begin{eqnarray*}}
\def\eeqars{\end{eqnarray*}}

\newcommand{\smod}[1]{\hspace{-1mm} \pmod{#1}}

\def \ds{\displaystyle}

\newcommand{\nn}{\mathbb{N}}
\newcommand{\zz}{\mathbb{Z}}
\newcommand{\qq}{\mathbb{Q}}
\newcommand{\cc}{\mathbb{C}}

\newcommand{\hh}{\mathbb{H}}

\allowdisplaybreaks

\begin{document}

\title{Eta quotients, Eisenstein series and Elliptic Curves}  

\author{Ay\c{s}e Alaca, \c{S}aban Alaca, Zafer Selcuk Aygin}

\maketitle

\markboth{AY\c{S}E ALACA, \c{S}ABAN ALACA, ZAFER SELCUK AYGIN}
{ETA QUOTIENTS, EISENSTEIN SERIES AND ELLIPTIC CURVES}

\begin{abstract}

We express all the newforms of  weight $2$ and levels $30$, $33$, $35$, $38$, $40$, $42$, $44$, $45$ as linear combinations of eta quotients and Eisenstein series, and list their corresponding strong Weil curves. 

Let $p$ denote a prime and $E (\zz_p)$ denote the the group  of algebraic points of an elliptic curve $E$ over $\zz_p$. 
We give a  generating function for the  order of $E (\zz_p)$ for certain strong Weil curves in terms of eta quotients and Eisenstein series. 
We then use our generating functions to deduce congruence relations  for the order of $E (\zz_p)$  for those  strong Weil curves.

\vspace{2mm}


\vspace{2mm}

\end{abstract}

\section{Introduction}

Let $\nn$, $\zz$, $\qq$ and $\cc$ denote the sets of positive integers, integers, rational numbers and complex numbers, respectively. 
Let $N\in\nn$ and $k \in \zz$. Let $\Gamma_0(N)$ be the modular subgroup defined by
\beqars
\Gamma_0(N) = \left\{ \left(
\begin{array}{cc}
a & b \\
c & d
\end{array}
\right)  \mid  a,b,c,d\in \zz ,~ ad-bc = 1,~c \equiv 0 \smod {N}
\right\} .
\eeqars 
We write $M_k(\Gamma_0(N))$ to denote the space of modular forms of weight $k$ and level~$N$. 
The Dedekind eta function $\eta (z)$ is the holomorphic function defined on the upper half plane $\hh = \{ z \in \cc \mid \mbox{\rm Im}(z) >0 \}$ 
by the product formula
\beqars
\eta (z) = e^{\pi i z/12} \prod_{n=1}^{\infty} (1-e^{2\pi inz}).
\eeqars
A product of the form
\beqar
f(z) = \prod_{1\leq \delta \mid N} \eta^{r_{\delta}} ( \delta z) ,
\eeqar 
where $r_{\delta} \in \zz$, not all zero, is called an eta quotient.  
We set $q:=q(z)=e^{2 \pi i z}$.  We also set $[n]f(z):=a_n$
for $ \ds f(z)=\sum_{n\in \zz} a_n q^{ n}$. 

Martin and Ono  \cite{ono} listed all the weight $2$ newforms that are eta quotients, and gave their corresponding strong Weil curves. There are such eta quotients only for levels $11$, $14$, $15$, $20$, $24$, $27$, $32$, $36$, $48$, $64$, $80$, $144$. 

In this paper we express all the newforms of weight $2$ and levels $30$, $33$, $35$, $38$, $40$, $42$, $44$, $45$ as linear combinations of eta quotients and Eisenstein series, and give their corresponding strong Weil curves. 

Let $p$ denote a prime and $E (\zz_p)$ denote the the group  of algebraic points of an elliptic curve $E$ over $\zz_p$. 
We give a  generating function for the  order of $E (\zz_p)$ for certain strong Weil curves in terms of eta quotients and Eisenstein series. 
We then use our generating functions to deduce congruence relations  for the order of $E (\zz_p)$  for those  strong Weil curves.

\section{Preliminary results}

Appealing to \cite[Theorem 1.64, p. 18]{onoweb} and  \cite[Corollary 2.3, p. 37]{Kohler} (see also  \cite{Ligozat, Kilford, alacaaygin}) 
one can show that 
\beqar
 \displaystyle \frac{\eta(z)\eta(3z)\eta^{3}(10z)\eta^{3}(30z)}{\eta(2z)\eta(5z)\eta(6z)\eta(15z)} \in M_2(\Gamma_0(30)), \\
 \displaystyle \frac{\eta^{3}(3z)\eta^{3}(33z)}{\eta(z)\eta(11z)}, \eta^{2}(3z)\eta^{2}(33z), \eta(z)\eta(3z)\eta(11z)\eta(33z) \in  M_2(\Gamma_0(33)), 
 \nonumber \\
 \displaystyle \frac{\eta^3(5z)\eta^3(7z)}{\eta(z)\eta({35z})} , \frac{\eta^3(z)\eta^3({35z})}{\eta(5z)\eta({7z})} \in M_2(\Gamma_0(35)), 
 \nonumber \\
\frac{\eta^4(2z)\eta^4({38z})}{\eta^2(z)\eta^2({19z})}, \frac{\eta^3(2z)\eta^3({19z})}{\eta(z)\eta({38z})} , \frac{\eta^3(z)\eta^3({38z})}{\eta(2z)\eta({19z})},  \frac{\eta^4(z)\eta^4({19z})}{\eta^2(2z)\eta^2({38z})} \in   M_2(\Gamma_0(38)),
 \nonumber \\
  \frac{\eta^{2}(z)\eta^{2}(5z)\eta^{2}(8z)\eta^{2}(40z)}{\eta(2z)\eta(4z)\eta(10z)\eta(20z)} \in M_2(\Gamma_0(40)),
 \nonumber \\
 \frac{\eta^{2}({2z}\eta^{2}({3z}))\eta^{2}({14z})\eta^{2}({21z})}{\eta(z)\eta({6z})\eta({7z})\eta({42z})}, \frac{\eta^{2}(z)\eta^{2}({6z})\eta^{2}({7z})\eta^{2}({42z})}{\eta({2z})\eta({3z})\eta({14z})\eta({21z})} \in M_2(\Gamma_0(42)),
 \nonumber \\
 \frac{\eta^{3}({4z})\eta^{3}({11z})}{\eta(z)\eta({44z})}, \frac{\eta^{3}(z)\eta^{3}({44z})}{\eta({4z})\eta({11z})} \in  M_2(\Gamma_0(44)),
 \nonumber \\
 \frac{\eta({3z})\eta^{2}({5z})\eta^{2}({9z})\eta({15z})}{\eta(z)\eta({45z})}, \frac{\eta^{2}(z)\eta({3z})\eta({15z})\eta^{2}({45z})}{\eta({5z})\eta({9z})}
\in M_2(\Gamma_0(45)),
 \nonumber
\eeqar 

The Eisenstein series $L(z)$ is defined as 
\beqars
L(z):=-\frac{1}{24}+\sum_{n>0} \sigma(n) q^{ n },
\eeqars
where $\ds \sigma(n)=\sum_{m\vert n} m$ is the sum of divisors function. By \cite[Theorem 5.8]{stein} we have
\beqar
L_t(z):=L(z)-tL(tz)\in M_2(\Gamma_0(N)) \mbox{~for all $0<t \mid N$}.
\eeqar
\noindent
We note that if $ f(z), ~ g(z) \in M_2(\Gamma_0(N))$, then for all $a,b\in \cc$, we have 
\beqar
af(z) + bg(z) \in M_2(\Gamma_0(N)).
\eeqar
We use the Sturm bound $S(N)$  to show the equality of two modular forms in the same modular  space. 
We just need  to check the equality of the  first $S(N) +1$ coefficients of their  Fourier series expansions.  
The following theorem is a special case for $M_2(\Gamma_0(N))$, see \cite[Theorem 3.13]{Kilford} for a general case.

\begin{theorem}
Let $f(z), ~g(z)\in M_2(\Gamma_0(N))$ with the Fourier series expansions\\
$\ds f(z)=\sum_{n=0}^{\infty} a_n q^{ n }$ and $\ds g(z)=\sum_{n=0}^{\infty} b_n q^{ n }$.
The Sturm bound $S(N)$ for the modular space $M_2(\Gamma_0(N))$ is given by
\beqars
\ds S(N)=\frac{N}{6} \ds \prod_{p|N} \big( 1+1/p \big),  
\eeqars
and so if $\ds a_n=b_n$ for all $ n \leq S(N)$
then $f(z) = g(z)$. 
\end{theorem}

Using Theorem 2.2 we calculate $S(N)$ for $N\in \{30, 33, 35, 38, 40, 42, 44, 45 \}$ as 
\beqars
&& S(30)=12,~S(33)=8,~S(35)=8,~S(38)=10,\\
&& S(40)=12,~S(42)=16,~S(44)=12,~S(45)=12.
\eeqars


\section{Main Results}

\begin{theorem}
Let $N \in \{ 30, 33, 35, 38, 40, 42, 44, 45\}$. In Table {\rm 3.1} below we express 
all the newforms $F_N(z)$ in $M_2(\Gamma_0(N))$  as linear combinations of eta quotients and Eisenstein series. 
\renewcommand{\arraystretch}{2.4}
\begin{longtable}{l l l l} 
\caption*{\hspace{-20mm}{\rm \normalsize Table  3.1: Weight $2$ newforms of levels $30, 33, 35, 38, 40, 42, 44, 45$} } \\
\hline \hline
Level & Name &  The eta quotients and Eisenstein series \\ 
\hline \hline
\endfirsthead
\hline
\hline 
Level & Name & The eta quotients and Eisenstein series \\ 
\hline \hline
\endhead
\hline 
\endfoot
\hline
\endlastfoot
$30$ & $F_{30}(z)=$ & $\displaystyle 6\frac{\eta(z)\eta(3z)\eta^{3}(10z)\eta^{3}(30z)}{\eta(2z)\eta(5z)\eta(6z)\eta(15z)}$ \\
 & & $\ds +2L_{2}(z)+L_{3}(z)+\frac{1}{5}L_{5}(z)-2 L_{6}(z)-\frac{2}{5}L_{10}(z)$ \\
 & & $\ds -\frac{1}{5}L_{15}(z)+\frac{2}{5}L_{30}(z)$  \\
\hline
$33$ &  $F_{33}(z)=$ &  $\displaystyle -10\frac{\eta^{3}(3z)\eta^{3}(33z)}{\eta(z)\eta(11z)}-6 \eta^{2}(3z)\eta^{2}(33z)$\\
&& $\ds -2\eta(z)\eta(3z)\eta(11z)\eta(33z) +\frac{1}{3}L_3(z)+L_{11}(z)-\frac{1}{3} L_{33}(z)$\\
\hline
$35$ &  $F_{35}(z)=$ & $\displaystyle 3\frac{\eta^3(5z)\eta^3(7z)}{\eta(z)\eta({35z})} -\frac{\eta^3(z)\eta^3({35z})}{\eta(5z)\eta({7z})}$ \\ & & $\ds +\frac{4}{5}L_5(z)-\frac{5}{7}L_7(z) -\frac{73}{35}L_{35}(z)$\\
\hline
$38$ & $\ds F_{38A}(z)=$ & { $ \displaystyle \frac{3}{7}\frac{\eta^3(z)\eta^3({38z})}{\eta(2z)\eta({19z})} -\frac{18}{7}\frac{\eta^4(2z)\eta^4({38z})}{\eta^2(z)\eta^2({19z})} -\frac{3}{7}\frac{\eta^3(2z)\eta^3({19z})}{\eta(z)\eta({38z})}$}\\
&&  $ \ds  -\frac{9}{28}\frac{\eta^4(z)\eta^4({19z})}{\eta^2(2z)\eta^2({38z})} -\frac{1}{7}L_2(z)
+\frac{1}{7}L_{19}(z)+\frac{1}{7}L_{38}(z) $  \\[1mm]
\hline
$38$ & $F_{38B}(z)=$ & { $\displaystyle \frac{\eta^3(2z)\eta^3({19z})}{\eta(z)\eta({38z})} 
+\frac{\eta^3(z)\eta^3({38z})}{\eta(2z)\eta({19z})} $}\\[1mm]
\hline
$40$ & $F_{40}(z)=$ & $\displaystyle -4\frac{\eta^{2}(z)\eta^{2}(5z)\eta^{2}(8z)\eta^{2}(40z)}{\eta(2z)\eta(4z)\eta(10z)\eta(20z)} 
+\frac{3}{2}L_2(z)+\frac{3}{2}L_4(z)+L_5(z)$\\
&& $\ds -L_8(z)-\frac{3}{2}L_{10}(z)-\frac{3}{2}L_{20}(z)+L_{40}(z)$\\
\hline
$42$ & $F_{42}(z)=$ & $\ds -8\frac{\eta(2z)\eta(3z)\eta^{2}(7z)\eta^{2}(42z)}{\eta(z)\eta(6z)} -8\frac{\eta(z)\eta(6z)\eta^{2}(14z)\eta^{2}(21z)}{\eta(2z)\eta(3z)}$\\
&&$\ds + L_2(z)-L_3(z)+L_6(z)+\frac{1}{7}L_7(z)-\frac{1}{7}L_{14}(z)$ \\
&& $\ds +\frac{1}{7}L_{21}(z)-\frac{1}{7}L_{42}(z)$\\
\hline
$44$ & $F_{44}(z)=$ & $\ds 3\frac{\eta^{3}({4z})\eta^{3}({11z})}{\eta(z)\eta({44z})}-3\frac{\eta^{3}(z)\eta^{3}({44z})}{\eta({4z})\eta({11z})}-2L_4(z)$ \\
&& $\ds +2L_{11}(z)-2L_{44}(z)$\\
\hline
$45$ & $F_{45}(z)=$ & $\ds 2\frac{\eta({3z})\eta^{2}({5z})\eta^{2}({9z})\eta({15z})}{\eta(z)\eta({45z})} -2\frac{\eta^{2}(z)\eta({3z})\eta({15z})\eta^{2}({45z})}{\eta({5z})\eta({9z})} $\\
&&$\ds +L_5(z)-\frac{2}{3}L_9(z)-\frac{2}{5}L_{15}(z)-\frac{14}{15}L_{45}(z)$\\
\end{longtable}
\end{theorem}

\begin{proof}
In \cite[Table 3]{Cremona} each newform of weight $2$ and level less than $1000$ has been given by listing its Fourier coefficients for primes up to $100$. Using the results from  \cite[p. 25]{Cremona} together with \cite[Table 3]{Cremona} we determine the first $S(N) +1$ terms of the 
Fourier series expansions of all the newforms of  weight $2$  and levels $N= 30$, $33$, $35$, $38$, $40$, $42$, $44$, $45$. 
We give them in Table 3.2 below.

\renewcommand{\arraystretch}{1.4}
\begin{longtable}{l l l l} 
\caption{First $S(N) +1$ terms of the Fourier series expansions of the  newforms of weight $2$ and level $N$} \\
\hline \hline
 $N$ & First $S(N)+1$ terms of the newforms of level $N$ \\ 
\hline \hline 
\endfirsthead
\hline \hline
\endhead
\hline 
\endfoot
\hline
\endlastfoot
 ${30}$ & $q - q^{2 }+q^{ 3 } +q^{  4 } -q^{ 5} -q^{6} -4q^{ 7 } -q^{ 8 }  +q^{ 9 } +q^{ 10 }+q^{ 12 } +O(q^{ 13 })$, \\
 ${33}$ & $q+ q^{2} -q^{3} -q^{4} -2q^{5} -q^{6} +4q^{7}-3q^{ 8 }+O( q^{9})$, \\
 ${35}$ & $q+ q^{3} -2q^{4} -q^{5} +q^{7} +O(q^{9})$, \\
 ${38A}$ & $q- q^{2} +q^{3} +q^{4} -q^{6} -q^{7} -q^{8} -2q^{9} +O(q^{11})$, \\
 ${38B}$ & $q+ q^{2} -q^{3} +q^{4} -4q^{5} -q^{6} +3q^{7} +q^{8} -2q^{9} -4q^{ 10 }+O(q^{11})$,\\
 ${40}$ & $q +q^{5} -4q^{7} -3q^{9} +4q^{11} +O(q^{13})$, \\
 ${42}$ &  $q+ q^{2} -q^{3} +q^{4} -2q^{5} -q^{6} -q^{7} +q^{8} +q^{9} -2q^{10} -4q^{11} -q^{12}$\\
 &\hspace{2mm} $+6q^{13} -q^{14} +2q^{15}+q^{ 16 } +O(q^{17})$,\\
${44}$ & $q +q^{3} -3q^{5} +2q^{7} -2q^{9} -q^{11} +O(q^{13})$,\\
${45}$ & $q +q^{2} -q^{4} -q^{5} -3q^{8} -q^{10} +4q^{11} +O(q^{13})$.
\end{longtable}

\noindent 
Let us consider the function $F_{30}(z)$ from Table 3.1. 
By (2.1)--(2.3), we have $F_{30}(z)\in M_2(\Gamma_0(30))$. 
Using MAPLE we calculate the first $S(30)+1=13$ terms of the Fourier series expansion of $F_{30}(z)$ as
\beqars
F_{30}(z)=q - q^{2 }+q^{ 3 } +q^{  4 } -q^{ 5} -q^{6} -4q^{ 7 } -q^{ 8 }  +q^{ 9 } +q^{ 10 } +q^{ 12 } +O(q^{ 13 }).
\eeqars
Since the first $13$ terms of the  newform of $M_2(\Gamma_0(30))$ in Table 3.2 are the  same as the first $13$ terms of $F_{30}(z)$  
then it follows from Theorem 2.1 that $F_{30}(z)$ is equal to the  newform of level $30$. 
The remaining cases can be proven similarly. 
Note that there are two different weight $2$  newforms for level $38$, 
and following the notation from  \cite[Table 3]{Cremona} we label them as $38A$ and $38B$ in Table 3.2, 
which correspond to $F_{38A}(z)$ and $F_{38B}(z)$ in Table~3.1, respectively.
\end{proof}

\section{Some arithmetic properties of  $| E (\zz_p)|$ for certain elliptic curves}

We recall that $p$ denotes a prime. We first state the $a_p$ version of the Modularity Theorem, see \cite[Theorem 8.8.1]{DiamondShurman}, 
which gives a relation between $| E (\zz_p)|$ and the Fourier coefficients of the corresponding newform. 

\begin{theorem} {\rm (Modularity Theorem, Version $a_p$)} 
Let $a_1, a_2, a_3,a_4, a_6 \in \zz$. Let E be an elliptic curve over $\qq$ with conductor $N$ given by
\beqars
y^2 + a_1xy + a_3y = x^3 + a_2x^2 + a_4x + a_6, 
\eeqars
and let
\beqars
 E(\zz_p):=\{\infty\}\cup \{ (x,y)\in \zz_p \times \zz_p \vert y^2 + a_1xy + a_3y = x^3 + a_2x^2 + a_4x + a_6 \} .
\eeqars
Then for some newform $f\in S_2(\Gamma_0(N))$, we have
\beqars
[p]f_E(z)=p+1- | E(\zz_p) |  \mbox{ for $p\nmid N$}.
\eeqars
\end{theorem}


We deduce  the following theorem from \cite[Table 1]{Cremona}.

\begin{theorem}
Table {\rm 4.1} below is a list of elliptic curves, more specifically strong Weil curves, corresponding to the newforms given in Table {\em 3.1}.

\renewcommand{\arraystretch}{1.4}
\begin{longtable}{c c r r r r r}
\caption*{\rm Table 4.1: $y^2 + a_1xy + a_3y = x^3 + a_2x^2 + a_4x + a_6$} \\
\hline \hline
$Newform$ & Strong Weil curve & $a_1$ & $a_2$ & $a_{3}$ & $a_4$ & $a_6$\\  
\hline \hline 
\endfirsthead
\hline
$Newform$ & Elliptic curve & $a_1$ & $a_2$ & $a_{3}$ & $a_4$ & $a_6$\\ 
\hline
\endhead
\hline 
\endfoot
\hline
\endlastfoot
$ F_{30}(z) $ & $ E_{30A} $ & $ 1 $ & $ 0 $ & $ 1 $ & $ 1 $ & $ 2 $\\
$ F_{33}(z) $ & $ E_{33A} $ & $ 1 $ & $ 1 $ & $ 0 $ & $ -11  $ & $ 0 $\\
$ F_{35}(z) $ & $ E_{35A} $ & $ 0 $ & $ 1 $ & $ 1 $ & $ 9  $ & $ 1  $\\
$ F_{38A}(z) $ & $ E_{38A} $ & $ 1 $ & $ 0 $ & $ 1 $ & $ 9  $ & $ 90  $\\
$ F_{38B}(z) $ & $ E_{38B} $ & $ 1 $ & $ 1 $ & $ 1 $ & $ 0  $ & $ 1  $\\
$ F_{40}(z) $ & $ E_{40A} $ & $ 0 $ & $ 0 $ & $ 0 $ & $ -7  $ & $ -6  $\\
$ F_{42}(z) $ & $ E_{42A} $ & $ 1 $ & $ 1 $ & $ 1 $ & $ -4  $ & $ 5 $\\
$ F_{44}(z) $ & $ E_{44A} $ & $ 0 $ & $ 1 $ & $ 0 $ & $ 3  $ & $ -1  $\\
$ F_{45}(z) $ & $ E_{45A} $ & $ 1 $ & $ -1 $ & $ 0 $ & $ 0  $ & $ -5  $
\end{longtable}
\end{theorem}

We use Theorems 3.1, 4.1 and 4.2 to give  generating functions  
for the  group order  $| E_{N}(\zz_p) |$ of  elliptic curves  in Table 4.1.

\begin{theorem} 
Consider the elliptic curves listed in Table {\rm 4.1}. 
We have
\beqars
&& | E_{30A}(\zz_p) | = -6[p]\left(\frac{\eta(z)\eta(3z)\eta^{3}(10z)\eta^{3}(30z)}{\eta(2z)\eta(5z)\eta(6z)\eta(15z)}\right)  \mbox{~for all $p\nmid 30$,}\\
&& | E_{33A}(\zz_p) | =2[p] \left(5\frac{\eta^{3}(3z)\eta^{3}(33z)}{\eta(z)\eta(11z)}+ \eta(z)\eta(3z)\eta(11z)\eta(33z)\right) \mbox{~for all $p\nmid 33$,}\\ 
&& | E_{35A}(\zz_p) | =   3(p+1)-[p]\left(3\frac{\eta^3(5z)\eta^3(7z)}{\eta(z)\eta({35z})} -\frac{\eta^3(z)\eta^3({35z})}{\eta(5z)\eta({7z})}  \right) \mbox{~for all $p\nmid 35$,}\\
&& | E_{38A}(\zz_p) | =  \frac{6}{7}(p+1)- \frac{3}{7}[p]\left( \frac{\eta^3(z)\eta^3({38z})}{\eta(2z)\eta({19z})} - 6\frac{\eta^4(2z)\eta^4({38z})}{\eta^2(z)\eta^2({19z})} \right. \\
&& \hspace{42mm} \left. - \frac{\eta^3(2z)\eta^3({19z})}{\eta(z)\eta({38z})} -\frac{3}{4}\frac{\eta^4(z)\eta^4({19z})}{\eta^2(2z)\eta^2({38z})}  \right) \mbox{~for all $p\nmid 38 $,}\\ 
&& | E_{38B}(\zz_p) | =  (p+1)-[p]\left( \frac{\eta^3(2z)\eta^3({19z})}{\eta(z)\eta({38z})} +\frac{\eta^3(z)\eta^3({38z})}{\eta(2z)\eta({19z})}  \right) \mbox{~for all $p\nmid 38$,}\\ 
&& | E_{40A}(\zz_p) | = 4[p] \left(\frac{\eta^{2}(z)\eta^{2}(5z)\eta^{2}(8z)\eta^{2}(40z)}{\eta(2z)\eta(4z)\eta(10z)\eta(20z)}\right) 
\mbox{~for all $p\nmid 40$,}\\ 
&& | E_{42A}(\zz_p) | = 8[p]  \left(\frac{\eta(2z)\eta(3z)\eta^{2}(7z)\eta^{2}(42z)}{\eta(z)\eta(6z)}+\frac{\eta(z)\eta(6z)\eta^{2}(14z)\eta^{2}(21z)}{\eta(2z)\eta(3z)}\right) \\
&&\hspace{115mm} \mbox{~for all $p\nmid 42$.} \\
&& | E_{44A}(\zz_p) | = 3(p+1)- 3[p]\left(  \frac{\eta^{3}({4z})\eta^{3}({11z})}{\eta(z)\eta({44z})}- \frac{\eta^{3}(z)\eta^{3}({44z})}{\eta({4z})\eta({11z})} \right) \mbox{~for all $p\nmid 44$,}\\ 
&& | E_{45A}(\zz_p) | =   2(p+1)- 2[p]\left( \frac{\eta({3z})\eta^{2}({5z})\eta^{2}({9z})\eta({15z})}{\eta(z)\eta({45z})} \right.\\
&&\hspace{60mm} \left. - \frac{\eta^{2}(z)\eta({3z})\eta({15z})\eta^{2}({45z})}{\eta({5z})\eta({9z})} \right)   \mbox{~for all $p\nmid 45$,}
\eeqars
\end{theorem}

\begin{proof}
We just prove the first and the last  equalities as the remaining ones can be proven similarly. 
By Theorems 3.1, 4.1 and 4.2, for all $p\nmid 30$, we have, 
\beqars
&& \hspace{-10mm}  | E_{30A}(\zz_p) | =  p+1-[p]F_{30A}(z) \\
&&  = p+1- [p]\Big( 6\frac{\eta(z)\eta(3z)\eta^{3}(10z)\eta^{3}(30z)}{\eta(2z)\eta(5z)\eta(6z)\eta(15z)} \\
&&\hspace{8mm}  +2L_{2}(z)+L_{3}(z)+\frac{1}{5}L_{5}(z)-2 L_{6}(z)-\frac{2}{5}L_{10}(z)-\frac{1}{5}L_{15}(z)+\frac{2}{5}L_{30}(z)\Big)\\
&&=p+1-6 [p]\Big( \frac{\eta(z)\eta(3z)\eta^{3}(10z)\eta^{3}(30z)}{\eta(2z)\eta(5z)\eta(6z)\eta(15z)} \Big)-\sigma(p)\\
&&= -6 [p]\Big( \frac{\eta(z)\eta(3z)\eta^{3}(10z)\eta^{3}(30z)}{\eta(2z)\eta(5z)\eta(6z)\eta(15z)} \Big),
\eeqars
which completes the proof of the first equality. 

By Theorems 3.1, 4.1 and 4.2, for all $p\nmid 45$, we have  
\beqars
&&\hspace{-5mm}  | E_{45A}(\zz_p) |  =  p+1-[p]F_{45}(z) \\
&&  = p+1- [p]\Big( 2\frac{\eta({3z})\eta^{2}({5z})\eta^{2}({9z})\eta({15z})}{\eta(z)\eta({45z})} -2\frac{\eta^{2}(z)\eta({3z})\eta({15z})\eta^{2}({45z})}{\eta({5z})\eta({9z})} \\
&&\hspace{8mm} +L_5(z)-\frac{2}{3}L_9(z)-\frac{2}{5}L_{15}(z)-\frac{14}{15}L_{45}(z) \Big)\\
&& = p+1-2 [p]\Big( \frac{\eta({3z})\eta^{2}({5z})\eta^{2}({9z})\eta({15z})}{\eta(z)\eta({45z})} -\frac{\eta^{2}(z)\eta({3z})\eta({15z})\eta^{2}({45z})}{\eta({5z})\eta({9z})} \Big) 
+\sigma(p)\\
&& = 2(p+1)-2 [p]\Big(  \frac{\eta({3z})\eta^{2}({5z})\eta^{2}({9z})\eta({15z})}{\eta(z)\eta({45z})} -\frac{\eta^{2}(z)\eta({3z})\eta({15z})\eta^{2}({45z})}{\eta({5z})\eta({9z})} \Big), 
\eeqars
which completes the proof of the last equality.
\end{proof}

The following congruence relations follow immediately from Theorem 4.2.

\begin{corollary} We have
\beqars
&& | E_{30A}(\zz_p) | \equiv 0 \pmod 6 \mbox{~for all $p\nmid 30$,}\\
&& | E_{33A}(\zz_p) | \equiv 0 \pmod 2 \mbox{~for all $p\nmid 33$,}\\ 
&& | E_{38A}(\zz_p) | \equiv 0 \pmod 3 \mbox{~for all $p\nmid 38$,}\\ 
&& | E_{40A}(\zz_p) | \equiv 0 \pmod 4 \mbox{~for all $p\nmid 40$,}\\ 
&& | E_{42A}(\zz_p) | \equiv 0 \pmod 8 \mbox{~for all $p\nmid 42$,}\\
&& | E_{44A}(\zz_p) | \equiv 0 \pmod 3 \mbox{~for all $p\nmid 44$,}\\
&& | E_{45A}(\zz_p) | \equiv 0 \pmod 2 \mbox{~for all $p\nmid 45$.}
\eeqars
\end{corollary}

\section{Remarks}

\noindent
{\bf (1)} The linear combinations of eta quotients and Eisenstein series in Table 3.1 have been chosen in a way that 
we can prove Theorem 4.3. In Table 5.1 below we give alternative representations for the newforms of levels $33, 40, 42$ for which 
have fewer number of eta quotients and Eisenstein series. 
We note that representation of the newform for level $33$ in \cite{pathak} is the same as the one in Table 5.1. 


\renewcommand{\arraystretch}{2.4}
\begin{longtable}{l l l l} 
\caption{Alternative representations for weight $2$ newforms of levels $33$, $40$, $42$} \\
\hline \hline
Level & Name &  The eta quotients and Eisenstein series \\ 
\hline \hline 
\endfirsthead
\hline
Level & Name & The eta quotients and Eisenstein series \\ 
\hline 
\endhead
\hline 
\endfoot
\hline
\endlastfoot
$33$ &  $F'_{33}(q)=$ &  $\displaystyle 3\eta^2(q^3)\eta^2(q^{33})+3\eta(q)\eta(q^3)\eta(q^{11})\eta(q^{33})+\eta^2(q)\eta^2(q^{11})$\\
\hline
$40$ & $F'_{40}(q)=$ & $\displaystyle 2\eta^2(q^2)\eta^2(q^{10})-\frac{\eta(q^2)\eta^2(q^{8})\eta^5(q^{20})}{\eta(q^4)\eta(q^{10})\eta^2(q^{40})}$\\[1mm]
\hline
$42$ & $F'_{42}(q)=$ & $\ds \frac{\eta^{2}(q^{2}\eta^{2}(q^{3}))\eta^{2}(q^{14})\eta^{2}(q^{21})}{\eta(q)\eta(q^{6})\eta(q^{7})\eta(q^{42})}- \frac{\eta^{2}(q)\eta^{2}(q^{6})\eta^{2}(q^{7})\eta^{2}(q^{42})}{\eta(q^{2})\eta(q^{3})\eta(q^{14})\eta(q^{21})}$
\end{longtable}

{\bf (2)} In \cite{ono} Martin and Ono represented weight $2$ newforms of levels $11, 14, 15, 20$, $24, 27, 32, 36, 48, 64, 80, 144$ 
in terms of single eta quotients. 
Let $F_{N}(z)$ be the newform(s) at level $N$. Using the arguments from this paper we 
give alternative representations for the new forms of levels $11, 14, 15$  in Table 5.2.

\renewcommand{\arraystretch}{2.4}
\begin{longtable}{l l l l} 
\caption{Alternative representations for weight $2$ newforms of levels $11$, $14$, $15$} \\
\hline \hline
Level & Name &  The eta quotients and Eisenstein series \\
\hline \hline
\endfirsthead
\hline
\hline 
Level & Name &  The eta quotients and Eisenstein series \\ 
\hline \hline
\endhead
\hline 
\endfoot
\hline
\endlastfoot
$11$ &  $F_{11}(z)=$ &  $\displaystyle  -5 \frac{\eta^4(2z)\eta^4(22z)}{\eta^2(z)\eta^2(11z)}-4 \eta^2(2z)\eta^2(22z) +\frac{1}{2} L_2(z)
\ds +L_{11}(z)-\frac{1}{2} L_{22}(z) $\\[1mm]
\hline
$14$ & $F_{14}(z)=$ & $\displaystyle \frac{6}{5}\frac{\eta^5(z)\eta^5(14z)}{\eta^3(2z)\eta^3(7z)}  +\frac{13}{5} L_2(z) -\frac{13}{5} L_7(z) + L_{14}(z)$\\[1mm]
\hline
$15$ & $F_{15}(z)=$ & $\displaystyle -4\frac{\eta^3(3z)\eta^3(15z)}{\eta(z)\eta(5z)}  +\frac{1}{3} L_3(z) + L_5(z) -\frac{1}{3} L_{15}(z) $\\[1mm] 
\end{longtable}

\noindent Corresponding strong Weil curves with conductors $11$, $14$ and $15$  are 
\beqars
&& E_{11A}: y^2+y=x^3- x^2-10 x -20 ,\\
&& E_{14A}: y^2+xy+y=x^3+4 x -6,\\
&& E_{15A}: y^2+xy+y=x^3+ x^2-10 x -10,
\eeqars
respectively, see \cite[Table 1]{Cremona} and \cite{ono}.  Similar to Theorem 4.3, we obtain 
\beqars
&& \mid E_{11A}(\zz_p) \mid =5[p] \left( \frac{\eta^4(2z)\eta^4(22z)}{\eta^2(z)\eta^2(11z)}\right) \mbox{~for all $p\nmid 11$,}\\
&& \mid E_{14A}(\zz_p) \mid  =-\frac{6}{5}[p] \left( \frac{\eta^5(z)\eta^5(14z)}{\eta^3(2z)\eta^3(7z)} \right) \mbox{~for all $p\nmid 14$,}\\
&& \mid E_{15A}(\zz_p) \mid =4 [p] \left(\frac{\eta^3(3z)\eta^3(15z)}{\eta(z)\eta(5z)} \right) \mbox{~for all $p\nmid 15$.}
\eeqars
Thus we deduce the congruence relations
\beqars
&& \mid E_{11A}(\zz_p) \mid \equiv 0 \pmod 5 \mbox{, for all $p\nmid 11$,}\\
&& \mid E_{14A}(\zz_p) \mid  \equiv 0 \pmod 6 \mbox{, for all $p\nmid 14$,}\\
&& \mid E_{15A}(\zz_p) \mid \equiv 0 \pmod 4 \mbox{, for all $p\nmid 15$.}
\eeqars

\newpage

{\bf (3)} 
Numerical results indicate that generating functions for $|E_N(\zz_p)|$ for a family of elliptic curves $E_N$ of conductor $N$ 
can be given by eta quotients and Eisenstein series. 
To this end, we have obtained generating functions for $|E_N(\zz_p)|$  for  elliptic curves $E_N$ of  conductor $N$ 
for various $N <100$, which will be a part of the PhD thesis of Zafer Selcuk Aygin.

\section*{Acknowledgments} 
The research of the first two authors was supported 
by Discovery Grants from the Natural Sciences and Engineering Research Council of Canada (RGPIN-418029-2013 and RGPIN-2015-05208). Zafer Selcuk Aygin's studies are supported by Turkish Ministry of Education.

\vspace{3mm}
\noindent
Centre for Research in Algebra and Number Theory \\
School of Mathematics and Statistics \\
Carleton University, Ottawa \\
Ontario, K1S 5B6, Canada \\

\noindent
AyseAlaca@cunet.carleton.ca\\
SabanAlaca@cunet.carleton.ca\\
ZaferAygin@cmail.carleton.ca

\end{document}